\documentclass[11pt, a4paper]{amsart}
\usepackage[utf8]{inputenc}
\usepackage{amsmath}
\usepackage{amsthm}
\usepackage{amssymb}
\usepackage{hyperref}
\usepackage{verbatim}

\newtheorem{theorem}{Theorem}
\newtheorem{lemma}{Lemma}
\newtheorem{prop}{Proposition}
\newtheorem{definition}{Definition}
\allowdisplaybreaks

\newcommand{\R}{\mathbb R}
\newcommand{\s}{\mathbb S}
\newcommand{\vol}{\operatorname{vol}}
\newcommand{\E}{\mathcal E}
\renewcommand{\L}{\mathcal L}
\newcommand{\bo}{\mathbb B}

\def \gradient{\nabla}
\newcommand{\dt}[1]{\frac{\partial #1}{\partial t}}
\def \pmas {{p}}

\def \gradient{\nabla}
\def \Lp{L_p}
\def \L2{L_2}
\def \gl {\operatorname{GL}}

\title[Affine trace inequality]{The sharp affine $L^2$ Sobolev trace inequality and variants}
\subjclass[2010]{ Primary 46E35, Secondary 46E39, 51M16}
\keywords{$\Lp$ Busemann-Petty centroid inequality, affine Sobolev inequalities, Sobolev trace inequality}

\author{P. L.  De N\'apoli}
\address{Pablo Luis De N\'apoli: IMAS (UBA-CONICET) and Departamento de Matem\'atica, Facultad de Ciencias Exactas y Naturales, Universidad de Buenos Aires, Ciudad Universitaria, 1428 Buenos Aires, Argentina.}
\email{pdenapo@dm.uba.ar}

\author{J. Haddad}
\address{Julián Haddad: Departamento de Matem\'atica, ICEx,  Universidade Federal de Minas Gerais, 30.123–970, Belo Horizonte, Brasil.}
\email{julianhaddad@ufmg.br}

\author{C. H. Jim\'enez}
\address{Carlos Hugo Jim\'enez:  Departamento de Matem\'atica, ICEx,  Universidade Federal de Minas Gerais, 30.123–970, Belo Horizonte, Brasil.}
\email{carloshugo@us.es}

\author{M. Montenegro}
\address{Marcos Montenegro: Departamento de Matem\'atica, ICEx,  Universidade Federal de Minas Gerais, 30.123–970, Belo Horizonte, Brasil.}
\email{montene@mat.ufmg.br}

\thanks{ The first author was supported by CONICET under grant PIP 1420090100230 and by Universidad de
	Buenos Aires under grant 20020120100029BA (Ubacyt). The second author was supported by CAPES (BJT 064/2013). The third author acknowledges the support provided by CAPES and IMPA. The fourth author was supported by CAPES (BEX 6961/14-2), CNPq (PQ 306406/2013-6) and Fapemig (PPM 00223-13).}

\begin{document}
\maketitle

\begin{abstract}
	We establish a sharp affine $L^p$ Sobolev trace inequality by using the $\Lp$ Busemann-Petty centroid inequality. For $p = 2$, our affine version is stronger than the famous sharp $L^2$ Sobolev trace inequality proved independently by Escobar and Beckner. Our approach allows also to characterize all cases of equality in this case. For this new inequality, no Euclidean geometric structure is needed.
\end{abstract}

\section{Introduction and statements}

Sobolev inequalities are among the fundamental tools connecting analysis and geometry. A specially important inequality, the sharp $L^2$ Sobolev trace inequality on the half-space $\R^n_+ := \R_+ \times \R^{n-1}$, $n \geq 3$, states that

\begin{equation} \label{E.1}
\left (\int_{\partial \R^n_+}|f(0,x)|^{\frac{2(n-1)}{n-2}}\; dx\right
)^{\frac{n-2}{n-1}}\leq \mathcal{K}_n \int_{\R^n_+}|\nabla f(t,x)|^2\; dx dt
\end{equation}
for all measurable functions $f: \R^n_+ \rightarrow \R$ with $L^2$-integrable gradient on $\R^n_+$, where $|\cdot|$ denotes the Euclidean norm in $\R^n$ in the right-hand side.

The existence of extremal functions for \eqref{E.1} was proved by Lions in \cite{Lions} by using concentration–compactness principle. The classification of them was established, independently, by Escobar in \cite{Escobar} and Beckner in \cite{Be} by using the conformal invariance satisfied by this inequality. The central tool used by Escobar is the Obata’s Rigidity Theorem, while Beckner reformulated the inequality dually in terms of fractional Sobolev spaces and used the sharp Hardy-Littlewood-Sobolev inequality. Precisely, equality in \eqref{E.1} holds if, and only if, $f(t,x)$ has the form

\[
f(t,x) = \gamma \left( (t + \delta)^2 + |x - x_0|^2\right)^{-\frac{n-2}{2}}
\]
for some real constants $\gamma$ and $\delta$ with $\delta > 0$ and a point $x_0 \in \R^{n-1}$. In particular, the value of the best constant $\mathcal{K}_n$ is given by

\[
\mathcal{K}_n = \frac 1 {\sqrt{\pi} (n-2)} \left(\frac{\Gamma (n)}{(n-1) \Gamma \left(\frac{n-1}{2}\right)}\right)^{\frac{1}{n-1}}\, .
\]

The $L^p$ version of the sharp Sobolev trace inequality states that

\begin{equation} \label{E.2}
\left (\int_{\partial \R^n_+}|f(0,x)|^{\frac{p(n-1)}{n-p}}\; dx\right
)^{\frac{n-p}{n-1}}\leq \mathcal{K}_{n,p} \int_{\R^n_+}|\nabla f(t,x)|^p\; dx dt
\end{equation}
for all measurable functions $f: \R^n_+ \rightarrow \R$ with $L^p$-integrable gradient on $\R^n_+$. In \cite{Na}, by using a mass transport method, Nazaret exhibited extremal functions to this inequality using an arbitrary norm of the gradient for any $1 < p < n$.

The literature related to \eqref{E.1} is extremely vast. Actually, its initial motivation came from differential geometry with the study of the Yamabe problem on compact manifolds with boundary by Escobar in \cite{E1, E2}, paving the way for a series of related works, see for example \cite{Br, HLi, M}, among other important references. More recently, the inequality \eqref{E.1} has played a relevant role in analysis, as for instance in the treatment of non-local elliptic boundary problems involving the square root of the Laplace operator on Euclidean domains. We list for example \cite{CT, CR, CaSi, CDDS, S} among the first references in a rich and growing literature.
Furthermore, we mention that the knowledge of minimizers and best constants for the trace inequalities is useful, for instance, in order to obtain existence results for non linear Steklov-like eigenvalue problems with critical growth in domains, see \cite{FBS}. Among other recent works on a slightly different setting we also mention for example \cite{Ci1,Ci2}.

An important inequality related to the sharp $L^p$ Sobolev trace inequality is the sharp $L^p$ Sobolev inequality established by Federer and Fleming \cite{FF}, Fleming
and Rishel \cite{FR} and Maz’ja \cite{Ma} for $p = 1$ and, independently, by Aubin \cite{Au} and Talenti \cite{Ta} for $1 < p < n$. A stronger version of that latter, the sharp {\it affine} $L^p$ Sobolev inequality, was introduced and proved by Zhang \cite{Z} for $p = 1$ and by Lutwak, Yang and Zhang \cite{L-Y-Z-1} for $1 < p < n$. We refer to \cite{HaSc1}, \cite{HaSc2} and \cite{HSX} for other recently obtained affine versions of classical inequalities.

In the present work, we prove a sharp affine $L^2$ Sobolev trace inequality on $\R^n_+$. The new trace inequality is significantly stronger than and implies the sharp $L^2$ Sobolev trace inequality of Escobar and Beckner. A remarkable feature of the new inequality is its independence of the norm chosen for the Euclidean space. In other words, such an inequality depends only on the vector space structure and Lebesgue measure of $\R^n$, so that it is invariant under all affine transformations of $\R^n_+$. It is surprising the existence of this new inequality because the sharp $L^2$ Sobolev trace inequality relies strongly on the Euclidean geometric structure of $\R^n$.

Isoperimetric inequalities are quite useful tools to compare parameters associated generally to convex bodies. In recent years a lot of attention has been paid to affine counterparts of the classical isoperimetric inequality on $\R^n$ that remain invariant under the action of affine transformations. Among these ones, it is known that the Petty projection inequality for $C^1$ domains is equivalent to the sharp affine $L^1$ Sobolev inequality, see \cite{Z}. For $1 < p < n$, the sharp affine $L^p$ Sobolev inequality was obtained by using essentially the $L^p$ extension of the Minkowski problem, P\'olya-Szeg\"o type results and the $L^p$ Petty projection inequality, see \cite{L-Y-Z-1}. The strong connection between the Petty projection and the Busemann-Petty (see \cite{Lut}) inequalities makes the use of the latter in this type of approach far from surprising, however, it provides us with some additional geometric information.

For the sharp affine $L^2$ Sobolev trace inequality, new ingredients are needed. The isoperimetric inequality needed is the $L^p$ Busemann-Petty centroid inequality firstly proved by Lutwak, Yang and Zhang \cite{L-Y-Z}. This inequality generalizes the classical Busemann-Petty centroid inequality due to Petty \cite{Petty} that compares the ratio between the volume of a convex body $K$ and that of its centroid body. A more recent proof was given in \cite{C-G} using the technique of shadow systems. The approach used here has been recently introduced in \cite{HJM} and consists in connecting Sobolev type inequalities for general norms to affine Sobolev type inequalities via the $L^p$ Busemann-Petty centroid inequality. Our method reveals in an explicit and elementary way the geometric nature behind affine inequalities.

In order to state our main theorem, a little bit of notation should be introduced.

Let $n \geq 3$. We shall denote each point of the whole space $\R^n$ or half-space $\R^n_+ = \R_+ \times \R^{n-1}$ by $(t,x)$ or $y$ and each point of the Euclidean sphere $\s^{n-2}$ in $\R^{n-1}$ by $\xi$.

For smooth functions $f(t,x)$ with compact support on $\R^n$ we denote by $\frac{\partial f}{\partial t}$ and $\tilde{\nabla} f$ respectively the partial derivative with respect to the variable $t$ and the gradient with respect to the variable $x$. Also $\nabla_\xi f$ stands for the directional derivative of $f$ with respect to the second variable in the direction of $\xi \in \s^{n-2}$, namely $\nabla_\xi f = \tilde{\nabla} f \cdot \xi$.

We introduce the new integral term

\begin{align*}
\E(f) :=\ &c_{n-1} \left(\int_{\s^{n-2}} \|\gradient_\xi f\|_{L^2(\R^n_+)}^{1-n} d \xi \right)^{\frac 1{1-n}} \\
=\ &c_{n-1} \left(\int_{\s^{n-2}} \left( \int_{\R^n_+} |\gradient_\xi f(y)|^2 dy \right)^{-\frac {n-1}2} d\xi \right)^{-\frac 1{n-1}},
\end{align*}
where

\begin{eqnarray*}
c_{n} = {n^{\frac{n+2}{2n}}} \sqrt[n]{\omega_n} \, .
\end{eqnarray*}
Here $\omega_n$ denotes the volume of the unit Euclidean ball $\bo^n$ in $\R^n$. It is known that
\[ \omega_n = \frac{\pi ^{\frac{n}{2}}}{\Gamma \left(\frac{n}{2}+1\right)}\, ,\]
where $\Gamma(\cdot)$ denotes as usual the Gamma function. 

The sharp affine $L^2$ Sobolev trace inequality on $\R^n_+$ states that

\begin{theorem} \label{mainthm}
Let $n \geq 3$. For any smooth function $f$ with compact support on $\R^n$, we have
\begin{equation}
\label{main_affine_trace_ineq}
\left (\int_{\partial \R^n_+}|f(0,x)|^{\frac{2(n-1)}{n-2}}\; dx\right
)^{\frac{n-2}{n-1}} \leq 2 \mathcal{K}_n \; \E(f) \left( \int_{\R^n_+} \left| \frac{\partial f}{\partial t} \right|^2 dx dt \right)^{\frac{1}{2}}\, ,
\end{equation}
where $\mathcal{K}_n$ is the best constant for the trace inequality (\ref{E.1}). Moreover, equality holds if, and only if,

\[
f(t,x) = \pm \left( (\lambda t + \delta)^2 + |B (x-x_0))|^2 \right)^ {-\frac{n-2}{2}}
\]
for some constants $\lambda, \delta > 0$, a point $x_0 \in \R^{n-1}$ and a matrix $B \in \gl_{n-1}$, where $\gl_{n-1}$ denotes the set of invertible real $(n-1) \times (n-1)$-matrices.
\end{theorem}

We remark that the inequality \eqref{main_affine_trace_ineq} is stronger than and implies the classical inequality \eqref{E.1}. In fact, following the ideas used in \cite{L-Y-Z-1}, applying the H\"older inequality and Fubini's theorem to $\E(f)$, one easily deduces that $\E(f) \leq \|\tilde{\nabla} f\|_2$. This inequality clearly implies the claim.

The sharp affine $L^2$ Sobolev trace inequality is indeed a particular case of an affine $L^p$ variant of Sobolev type trace inequalities.

Let $p > 1$. For smooth functions $f(t,x)$ with compact support on $\R^n$, denote

\begin{align*}
\E_p(f) :=\ &c_{n-1,p} \left(\int_{\s^{n-2}} \|\gradient_\xi f\|_{L^p(\R^n_+)}^{1-n} d \xi \right)^{\frac 1{1-n}} \\
=\ &c_{n-1,p} \left(\int_{\s^{n-2}} \left( \int_{\R^n_+} |\gradient_\xi f(y)|^p dy \right)^{-\frac {n-1}p} d\xi \right)^{-\frac 1{n-1}}\, ,
\end{align*}
where

\begin{eqnarray*}
c_{n,p} = \left(n \omega_n\right)^{\frac{1}{n}} \left(\frac{n \omega_n \omega_{p-1}}{2 \omega_{n+p-2}}\right){}^{\frac{1}{p}}\, .
\end{eqnarray*}
Here we recall that for a real number $s > 0$, one defines

\[
\omega_s = \frac{\pi ^{\frac{s}{2}}}{\Gamma \left(\frac{s}{2}+1\right)}.
\]

The following theorem is a natural $L^p$ extension of Theorem \ref{mainthm}:

\begin{theorem} \label{mainthm-p}
Let $1 < p < n$. For any smooth function $f$ with compact support on $\R^n$, we have
\begin{equation}
\label{main_affine_trace_ineq-p}
\left (\int_{\partial \R^n_+}|f(0,x)|^{\frac{p(n-1)}{n-p}}\; dx\right
)^{\frac{n-p}{n-1}} \leq p \mathcal{A}_{n,p} \; \E_p(f)^{p-1} \left( \int_{\R^n_+} \left| \frac{\partial f}{\partial t} \right|^p dx dt \right)^{\frac{1}{p}}\, ,
\end{equation}
where
\[
\mathcal{A}_{n,p} = \pi^{-\frac{p-1}{2}} \left(\frac{(p - 1)^{\frac{p-1}{p}}}{n - p}\right)^{p-1} \left(\frac{ \Gamma (n) \Gamma \left(\frac{n+1}{2}\right)}{(n-1)\Gamma \left(\frac{n-1}{p}\right) \Gamma \left(\frac{n (p-1)+1}{p}\right)}\right)^{\frac{p-1}{n-1}}
\]
is sharp for this inequality. Moreover, equality holds if

\[
f(t,x) = \pm \left( (\lambda t + \delta)^p + |B (x-x_0))|^p \right)^ {-\frac{n-p}{p(p-1)}}
\]
for some constants $\lambda, \delta > 0$, a point $x_0 \in \R^{n-1}$ and a matrix $B \in \gl_{n-1}$.
\end{theorem}

Note that the inequality \eqref{main_affine_trace_ineq-p} is invariant under affine transformations of $\R^n_+$. In precise terms, denote by $\gl_{n,+}$ the set of matrices of the form
\begin{equation}
\label{matrixform}
\left(
\begin{array}{cccc}
\lambda & 0 & \cdots & 0\\
0\\
\vdots & & B\\
0
\end{array}
\right)
\end{equation}
where $\lambda > 0$ and $B \in \gl_{n-1}$. Then, the inequality \eqref{main_affine_trace_ineq-p} is $\gl_{n,+}$ invariant. In particular, it does not depend on the
Euclidean structure of $\R_+^n$.

The paper is organized as follows. In section \ref{sec:not} we fix some notations to be used in the paper and present the $\Lp$ Petty-Projection inequality.
In section \ref{sec:lem} we prove two central tools in our method (Lemmas 1 and 2). These two lemmas and the $\Lp$ Petty-Projection inequality are just used in the proof of Theorem \ref{mainthm-p} which is done in section \ref{sec:proof}. Finally, in section \ref{sec:proof} we provide additional comments on the essence of our method and on the characterization of extremal functions stated in Theorem \ref{mainthm}. For convenience of the reader we include an appendix section devoted to more technical computations of a key constant ($d_1$) in the proof of Theorem \ref{mainthm-p}.

\section{Preliminaries on convex bodies}
\label{sec:not}

We recall that a convex body $K\subset\R^n$ is a convex compact subset of $\R^n$ with non-empty interior.

For $K\subset\mathbb R^n$ as before, its support function $h_K$ is defined as $$h_K(y)=\max\{\langle y, z\rangle\ :\ z\in K\}.$$ The support function, which describes the (signed) distances of supporting hyperplanes to the origin, uniquely characterizes $K$. We also have the gauge $\|\cdot\|_K$ and radial $r_K(\cdot)$ functions of $K$ defined respectively as
\[\|y\|_K:=\inf\{\lambda>0\ :\  y\in \lambda K\},\quad y\in\R^n\setminus\{0\},\] \[r_K(y):=\max\{\lambda>0\ :\ \lambda y\in K\},\quad y\in\R^n\setminus\{0\}.\]
Clearly, $\|y\|_K=\frac{1}{r_K(y)}$. We also recall that $\|\cdot\|_K$ it is actually a norm when the convex body $K$ is symmetric (i.e. $K=-K$). On the other hand, any centrally symmetric convex body $K$ is the unit ball for some norm in $\R^n$.

For $K\subset \R^n$ we define its polar body, denoted by $K^\circ$, by
\[K^\circ:=\{y\in\R^n\ :\ \langle y,z \rangle\leq 1\quad \forall z\in K\}.\]
Evidently, $h_K=r_{K^{\circ}}$. It is also easy to see that $(\lambda K)^\circ=\frac{1}{\lambda}K^\circ$ for $\lambda>0$. A simple computation using polar coordinates shows that
\[\vol(K)=\frac{1}{n}\int_{\s^{n-1}}r_K^n(y)dy=\frac{1}{n}\int_{\s^{n-1}}\|y\|^{-n}_K dy.\]

For a given convex body $K\subset\R^n$ we find in the literature many associated bodies to it, in particular Lutwak and Zhang introduced \cite{L-Z} for a body $K$ its $\Lp$-centroid body denoted by $\Gamma_pK$. This body is defined by

\[h_{\Gamma_pK}^p(y):=\frac{1}{a_{n,p}\vol(K)}\int_{K}|\langle y,z\rangle|^p dz\quad \mbox{ for }y\in\R^n,\]
where

\[ a_{n,p} = \frac{\omega_{n+p}}{\omega_2 \omega_n \omega_{p-1}}.\]

There are some other normalizations of the $\Lp$-centroid body in the literature, the previous one is made so that $\Gamma_p \bo^n=\bo^n$ for the unit ball in $\R^n$ centered at the origin.

Inequalities (usually affine invariant) that compare the volume of a convex body $K$ and that of an associated body are common in the literature. For the specific case of $K$ and $\Gamma_pK$, Lutwak, Yang and Zhang \cite{L-Y-Z} (see also \cite{C-G} for an alternative proof) came up with what it is known as the $\Lp$ Busemann-Petty centroid inequality, namely
\begin{equation}
\label{eqn:BPI}
\vol(\Gamma_pK)\geq \vol(K).
\end{equation}
This inequality is sharp if and only if $K$ is a $0$-symmetric ellipsoid. For a comprehensive survey on $\Lp$ Brunn-Minkowski theory and other topics within Convex Geometry we refer to \cite{Sch} and references therein.

\section{Fundamental lemmas}
\label{sec:lem}

Let $C:\R^n\rightarrow \R_+$ be an even convex function such that $C(x) > 0$ for $x \neq 0$. Assume also that $C$ is $q$-homogeneous with $q>1$, that is

\[ C(\lambda y)=\lambda^qC(y),\quad \forall \lambda\geq 0,\quad y\in\R^n.\]
Denote by $C^*$ its Legendre transform defined by

\[C^*(y) = \sup_{z \in \R^n} \{ \langle y, z \rangle - C(z) \}.\]
One knows that $C^*$ is also even, convex, $p$-homogeneous and positive for $x \neq 0$, where $\frac{1}{p}+\frac{1}{q}=1$. Equivalently, these assumptions on $C$ can be resumed by saying simply that $C$ is a power $q$ of an arbitrary norm on $\R^n$.

Let $C$ be a function as before. Denote by
\[K_C = \{y \in \R^n\ :\ C(y) \leq 1 \}. \]
It is easy to see that $K_C$ is a centrally symmetric convex body with non-empty interior in $\R^n$. Moreover, it is defined by the norm $\|y\|_{K_C}=C(y)^{\frac{1}{q}}$.

The starting point as one of the essential ingredients in our work is the following theorem of \cite{Na}.

\begin{theorem}
\label{nazaret}
Let $1 < p < n$ and $q > 1$ such that $\frac{1}{p} + \frac{1}{q} = 1$. Assume $C$ is even, convex, $q$-homogeneous and positive for $x \neq 0$. Then, for any smooth function $f$ with compact support on $\R^n$, we have
\begin{equation}
\label{trazaConNormaAbstracta}
\left (\int_{\partial \R^n_+}|f(0,x)|^{\frac{p(n-1)}{n-p}}\; dx\right
)^{\frac{n-p}{n-1}} \leq \mathcal{K}_{n,C} \int_{\R^n_+} C^*(\gradient f) dx dt.
\end{equation}
Moreover, equality holds in \eqref{trazaConNormaAbstracta} if

\begin{equation}
\label{extremalsC}
f(t,x) = \gamma \left( \frac {1}{ C(t+ \delta, x-x_0)} \right) ^ {\frac {n-p}{p} }
\end{equation}
for some real constants $\gamma$ and $\delta$ with $\delta > 0$ and a point $x_0 \in \R^{n-1}$.
\end{theorem}

The proof of Theorem \ref{mainthm-p} is based on two lemmas and uses Theorem \ref{nazaret} in a crucial manner with the appropriate choice of $C$ for each $f$, denoted by $C_f$, as stated in Definition \ref{Cfdefinition} below.

Before introducing $C_f$, we state the first key tool.

\begin{lemma}
\label{constant_L}
Let $p$, $q$ and $C$ be as in Theorem \ref{nazaret}. The optimal constant $\mathcal{K}_{n,C}$ is given by
\[\mathcal{K}_{n,C}
 = p^p (n-p)^{-\frac{p}{q}} (q(n-1))^{-\frac p {q(n-1)}} \left(\int_{(K_{C})_+} y_1^{q n - q - n} dy \right) ^{- \frac p {q(n-1) } },
 \]
where $(K_{C})_+ = K_C \cap \R^n_+$.

\end{lemma}
\begin{proof}
We calculate both left and right-hand sides of \eqref{trazaConNormaAbstracta} with $f$ given by \eqref{extremalsC}, $\gamma = \delta=1$ and $x_0 = 0$. Indeed, for

\[w_0(y) = C(y+e_1)^{-\frac {n-p}{q(p-1)} },\]
we have

\[\gradient w_0(y) = -\frac{n-p}{q(p-1)} C(y+e_1)^{-\frac {n-p}{q(p-1)}-1 } \gradient C(y+e_1). \]
On the right-hand side, we have

\begin{align*} C^*(\gradient w_0(y))
 = & \left|-\frac{n-p}{q(p-1)}\right|^pC(y+e_1)^{-\frac {p(n-p)}{q(p-1)}-p } C^*(\gradient C(y+e_1))\\
 = & (q-1)\left(\frac{n-p}{q(p-1)}\right)^p C(y+e_1)^{1-n}
\end{align*}
and on the other hand

\[
|w_0(y)|^{\frac{(n-1)p}{n-p}} = C(y+e_1)^{-\frac {n-p}{q(p-1)} \frac{(n-1)p}{n-p} } = C(y+e_1)^{1-n}.
\]

Let us consider the parametrization of $\R^n_+$ given by
\[
\left\{
\begin{array}{cc}
y = r \theta - e_1 \\
\theta \in \s^{n-1}_+, r \geq \theta_1^{-1}\\
\end{array}
\right.
\]
with Jacobian $r^{n-1}$.
Then,
\begin{align*}
\int_{\R^n_+} C^*(\gradient w_0(y)) dy &= (q-1)\left(\frac{n-p}{q(p-1)}\right)^p \int_{\s^{n-1}_+} \int_{ \theta_1^{-1}}^\infty r^{n-1} C(r \theta)^{1-n} dr d \theta \\
 &= (q-1)\left(\frac{n-p}{q(p-1)}\right)^p \int_{\s^{n-1}_+} \int_{ \theta_1^{-1}}^\infty r^{-q(n-1)+n-1} C(\theta)^{1-n} dr d \theta \\
 &= \frac{q-1}{q(n-1)-n}\left(\frac{n-p}{q(p-1)}\right)^p \int_{\s^{n-1}_+} \theta_1^{q(n-1)-n} C(\theta)^{1-n} d \theta.
\end{align*}

For the left-hand side of \eqref{trazaConNormaAbstracta}, we consider the parametrization of $\partial \R^n_+$,
\[
\left\{
\begin{array}{cc}
y = \theta_1^{-1} \theta  - e_1\\
\theta \in \s^{n-1}_+
\end{array}
\right.
\]
with Jacobian $\theta_1^{-n}$. Then,
\begin{align*}
\int_{\partial \R^n_+ } |w_0(y)|^{\frac{(n-1)p}{n-p}} dy
&= \int_{\s^{n-1}_+} \theta_1^{-n} C(\theta_1^{-1} \theta)^{1-n} d \theta \\
&= \int_{\s^{n-1}_+} \theta_1^{-n+q(n-1)} C(\theta)^{1-n} d \theta.
\end{align*}

Thus we have
\begin{align*}
\mathcal{K}_{n,C} &= \left(\frac{q-1}{q(n-1)-n}\left(\frac{n-p}{q(p-1)}\right)^p\right)^{-1} \left(\int_{\s^{n-1}_+} \theta_1^{-n+q(n-1)} C(\theta)^{1-n} d \theta\right) ^{-\frac{p}{q(n-1)}} \\
 &= p^p (n-p)^{-\frac{p}{q}} \left(\int_{\s^{n-1}_+} \theta_1^{-n+q(n-1)} C(\theta)^{1-n} d \theta\right) ^{-\frac{p}{q(n-1)}}
\end{align*}

Let us note now that
\begin{align*}
\int_{(K_{C})_+} y_1^{q n - q - n} dy &= \int_{\s^{n-1}_+} \int_0^{r_K(\theta)} r^{n-1} (r \theta_1)^{q n - q - n} d r d \theta\\
 &= \int_{\s^{n-1}_+} \theta_1^{q n - q - n}  \int_0^{r_K(\theta)} r^{q n - q - 1} dr d \theta\\
 &= \frac 1 {q(n-1)} \int_{\s^{n-1}_+} \theta_1^{q n - q - n} r_K(\theta)^{q n - q} d \theta = \frac 1 {q(n-1)} \int_{\s^{n-1}_+} \theta_1^{q n - q - n} C(\theta)^{1-n} d \theta.
\end{align*}

So we obtain that
\begin{align*}
\mathcal{K}_{n,C}
=& p^p (n-p)^{-\frac{p}{q}} \left( {q(n-1)} \int_{(K_{C})_+} y_1^{q n - q - n} dy \right) ^{-\frac{p}{q(n-1)}}\\
=& p^p (n-p)^{-\frac{p}{q}} (q(n-1))^{-\frac p {q(n-1)}} \left(\int_{(K_{C})_+} y_1^{q n - q - n} dy \right) ^{- \frac p {q(n-1) } }. \\
\end{align*}
\end{proof}

Throughout the remainder of paper we think of $\R^{n-1}$ as a subset  $\{0\} \times \R^{n-1}\subset \R^n$.

For each smooth function $f$ with compact support on $\R^n$, consider

\[L_f=\{\xi\in\R^{n-1}\ :\ \|\gradient_\xi f\|_\pmas\leq 1\},\]
which is a convex body in $\R^{n-1}$ defined by the norm $\|\xi\|_{L_f} = \|\gradient_\xi f\|_\pmas$.

For convenience, we set
\begin{equation}\label{eqn:Zf}
Z_p(f)=\left(\int_{\s^{n-2}}\|\gradient_\xi f\|_\pmas^{1-n}d\xi\right)^{\frac{1}{1-n}}
\end{equation}
and notice we have the identities

\begin{equation}
\label{eqn:VolZ}
(n-1) \vol(L_f) = Z_p(f)^{1-n}
\end{equation}
and
\begin{equation}
\label{eqn:Ez}
\E_p(f) = c_{n-1,p} Z_p(f).
\end{equation}

We now are ready to introduce the definition of the function $C_f$.

\begin{definition}
\label{Cfdefinition}
Let $f$ be a smooth function with compact support on $\R^n$. For $p > 1$ and $(t, x) \in \R^n_+$, we set

\[C_f^*(t, x) := \frac{\alpha_f}{p} |t|^p + \int_{\s^{n-2}} \|\gradient_\xi f\|_\pmas^{1-n-p} |\langle x, \xi \rangle |^p d \xi,\]
where $\alpha_f = \frac{p}{p-1}  \left\| \dt f \right\|_\pmas^{-p} Z_p(f)^{1-n}$. The function $C_f$ is defined as the Legendre transform of $C_f^*$.
\end{definition}
The specific choice of the constant $\alpha_f$ will be clarified at the last section.

\begin{prop}
\label{prop:Cdef}
Let $p$ and $f$ be as in Definition \ref{Cfdefinition}. The function $C_f^*$ is well defined, even, convex, $p$-homogeneous and $C_f^*(x) > 0$ for $x \neq 0$.
Thus its Legendre transform $C_f$ satisfies the hypotheses of Theorem \ref{nazaret}.
\end{prop}
\begin{proof}
The proof follows the same lines as in Proposition $1$ in \cite{HJM}.
\end{proof}

In order to simplify notation, for each $f$ as before, we denote $K_{C_f}$ by $K_f$, $(K_{C_f})_+$ by $(K_f)_+$, $\|\cdot\|_{K_{C_f}}$ by $\|\cdot\|_{K_f}$ and

\[
D_f^*(x) := \int_{\s^{n-2}} \|\gradient_\xi f\|_\pmas^{1-n-p} |\langle x, \xi \rangle |^p d \xi .
\]

It is easy to see that

\begin{equation} \label{productLegendre}
C_f(t,x) = \frac {\alpha_f^{1-q}}{q} |t|^q + D_f(x) ,
\end{equation}
where $D_f$ is the Legendre transform of $D_f^*$ and $q$ satisfies $\frac{1}{p} + \frac{1}{q} = 1$.

Let $K_{f,t} = \{x\in \R^{n-1}:\ (t,x) \in K_f\}$.

Using (\ref{productLegendre}) we see that $K_{f,0} = \{x \in \R^{n-1}:\ D_f(x) \leq 1\}$ and
\begin{equation}
\label{K0Size}
K_{f,t} = \left\{x \in \R^{n-1}:\ D_f(x) \leq 1- \frac{\alpha_f^{1-q}}{q} |t|^q \right\} = \left(1-\frac{\alpha_f^{1-q}}{q} |t|^q\right)^{\frac{1}{q}} K_{f,0}.
\end{equation}

The second central tool is stated as

\begin{lemma}
\label{GammaCentroid}
Let $p$ and $f$ be as in Definition \ref{Cfdefinition}. The relation between $K_{f,0}$ and $L_f$ is
\[K_{f,0} = \left((n+p-1) a_{n-1,p} \vol(L_f)\right)^{\frac{1}{p}} q^{\frac 1q}p^{\frac 1p} \Gamma_p L_f\]
 and as a consequence,
 \begin{equation}\label{eqn:VolKo}
vol(K_{f,0}) = \left(p q^{\frac pq } (n+p-1) a_{n-1,p}\right)^{\frac{n-1}{p}} \vol(L_f)^{\frac{n-1}p} \vol(\Gamma_p L_f),
 \end{equation}
where $\frac{1}{p} + \frac{1}{q} =1$.
\end{lemma}
\begin{proof}
Firstly, for $x \in \R^{n-1}$,

\begin{align*}
a_{n-1,p} \vol(L_f) h_{\Gamma_p L_f}^p(x) &= \int_{L_f} |\langle x, w \rangle|^p d w \\
&= \int_{\s^{n-2}} \int_0^{r_{L_f}(\xi)} r^{n-2}|\langle x, r \xi \rangle|^p dr d \xi\\
&= \int_{\s^{n-2}} |\langle x, \xi \rangle|^p \int_0^{r_{L_f}(\xi)} r^{n-2 + p}dr d\xi\\
&= \frac 1{n+p-1}\int_{\s^{n-2}} |\langle x, \xi \rangle|^p r_{L_f}(\xi)^{n-1+p} d \xi \\
&= \frac 1{n+p-1}\int_{\s^{n-2}} |\langle x, \xi \rangle|^p \| \gradient_\xi f\|_p^{1-n-p} d \xi \\
&= \frac 1{n+p-1}D_f^*(x)
\end{align*}
Thus
$$(n+p-1)a_{n-1,p} \vol(L_f) h_{\Gamma_p L_f}^p(x)=D_f^*(x),$$
so that

\begin{align*}
(\Gamma_p L_f)^\circ &= \{x \in \R^{n-1}:\ h_{\Gamma_p L_f}^p(x) \leq 1\} \\
&= ((n+p-1) a_{n-1,p} \vol(L_f))^{\frac{1}{p}}\{x \in \R^{n-1}:\ D_f^*(x) \leq 1\} \\
&= ((n+p-1) a_{n-1,p} \vol(L_f))^{\frac{1}{p}} q^{\frac 1q} p^{\frac 1p} (K_{f,0})^\circ,\\
\end{align*}
where in the last equality we used a general fact which we prove below for sake of completeness, once we were not able to find a proof in the literature. Finally, taking polar at both sides we obtain
\[
\Gamma_p L_f = ((n+p-1) a_{n-1,p} \vol(L_f))^{-\frac{1}{p}} q^{-\frac 1q} p^{-\frac 1p}  K_{f,0},\\
\]
and the proof of lemma follows.
\end{proof}

\begin{lemma}
Let $p > 1$ and $C:\R^m \to \R_+$, $m \geq 1$, be convex, $q$-homogeneous and positive for $x \neq 0$, where $\frac{1}{p}+\frac{1}{q}=1$. Then,
\[
K_{C^*} =  q^{\frac 1q} p^{\frac 1p} K_C^\circ .
\]
\end{lemma}
\begin{proof}
Firstly, denote $\ell_q := \max_{t \geq 0}\{t^{\frac{1}{q}}-t\} =  q^{-\frac{p}{q}} p^{-1}$.
It suffices to prove that $K_C^\circ = \{x \in \R^m: C^*(x) \leq \ell_q\}$ and the lemma follows.

Take $x \in K_C^\circ$.
Since $C$ is $q$-homogeneous, we have for any $y \in \R^m \setminus \{0\}$ that $C(y)^{-\frac{1}{q}} \; y \in K_C$, thus
\begin{align*}
\langle x, C(y)^{-\frac{1}{q}} \; y \rangle  & \leq 1\\
\langle x, y \rangle & \leq C(y)^{\frac{1}{q}}\\
\langle x,  y  \rangle -C(y) & \leq C(y)^{\frac{1}{q}} - C(y)\leq \ell_q,
\end{align*}
so that $C^*(x) \leq \ell_q$.

Now take $x \in \R^m \setminus K_C^\circ$. By definition of $K_C^\circ$, there exists $y \neq 0$ such that
\begin{align*}
\langle x, y \rangle &> 1 \geq C(y)\\
\langle x,  y  \rangle &> C(y)^{\frac{1}{q}}\\
\hbox{Hence, for any $t > 0$,} \ \  \langle x, t y  \rangle &> (t^q C(y)) ^{\frac{1}{q}}\\
\langle x, t y  \rangle - C(t y) &> ( t^q C(y) )^{\frac{1}{q}} - t^q C(y).
\end{align*}
Since $C(y) > 0$, we may choose $t > 0$ such that the right-hand side is maximized.
Thus, we obtain $C^*(x) > \ell_q$.
\end{proof}

\section{Proof of Theorem \ref{mainthm-p}}
\label{sec:proof}

We now prove Theorem \ref{mainthm-p} by using Lemmas \ref{constant_L} and \ref{GammaCentroid} and the $\Lp$ Petty-Projection inequality.

From the definition of $C_f$, we see that
\begin{align}
\label{mainthm_Cf_identity}
\int_{\R_+^n} C_f^*(\gradient f(y)) dy &= \frac{\alpha_f}{p} \int_{\R^n_+} \left| \frac{\partial f}{\partial t} \right|^p dx dt + \int_{\s^{n-2}} \|\gradient_\xi f\|_\pmas^{1-n} d\xi \\
&= \frac qp Z_p(f)^{1-n} + Z_p (f)^{1-n} \nonumber\\
&= q Z_p (f)^{1-n} \nonumber
\end{align}
where $\frac{1}{p} + \frac{1}{q} =1$.
From Lemma \ref{constant_L} and equation \eqref{K0Size}, we have

\begin{align}
\label{mainthm_LC_identity}
\mathcal{K}_{n,C_f} &= p^p (n-p)^{-\frac{p}{q}}(q(n-1))^{-\frac p {q(n-1)}}\left(\int_{(K_f)_+} y_1^{q n - q - n} dy \right) ^{- \frac p {q(n-1) } }  \\
&= p^p (n-p)^{-\frac{p}{q}}(q(n-1))^{-\frac p {q(n-1)}} \left( \int_0^{q^{\frac{1}{q}}\alpha_f^{\frac{1}{p}}} t^{q n - q - n} \left(1-\frac{\alpha_f^{1-q}}{q} t^q\right)^{\frac{n-1}{q}} \vol(K_{f,0}) dt \right) ^{- \frac p {q(n-1) } } \nonumber\\
&= p^p (n-p)^{-\frac{p}{q}} (q(n-1))^{-\frac p {q(n-1)}}\vol(K_{f,0})^{- \frac p {q(n-1) } } q^{-\frac pq} \alpha_f^{-1} \left( \int_0^1 t^{q n - q - n} \left(1- t^q\right)^{\frac{n-1}{q}} dt \right) ^{- \frac p {q(n-1) } } \nonumber\\
&= p^p (n-p)^{-\frac{p}{q}}(q(n-1))^{-\frac p {q(n-1)}} \left( \frac{1}{q}\frac{\Gamma\left(\frac{n-1}{p}\right)\Gamma\left(\frac{n-1}{q}+1\right)}{\Gamma(n)} \right)^{-\frac{p}{q(n-1)}} q^{-\frac 1q } \vol(K_{f,0})^{- \frac p {q(n-1) } } \alpha_f^{-\frac 1 {p}}, \nonumber
\end{align}
The latter using the fact that

\[\int_0^1 t^{q n - q - n} \left(1- t^q\right)^{\frac{n-1}{q}} dt=\frac{1}{q}\frac{\Gamma\left(\frac{n-1}{p}\right)\Gamma\left(\frac{n-1}{q}+1\right)}{\Gamma(n)}.\]
Putting together \eqref{mainthm_Cf_identity} and \eqref{mainthm_LC_identity}, we obtain

\begin{align*}
\mathcal{K}_{n,C_f} \int_{\R_+^n} C_f^*(\gradient f(y)) dy =&
\left(
p q^{-\frac{1}{pq}} (n-p)^{-\frac{1}{q}}(q(n-1))^{-\frac 1 {q(n-1)}} \left(\frac{\Gamma \left(\frac{n-1}{p}\right) \Gamma \left(\frac{n-1}{q}+1\right)}{q \Gamma (n)}\right)^{-\frac{1}{q(n-1)}}
\right)^p\\
\times& \vol(K_{f,0})^{- \frac p {q(n-1) } } \alpha_f^{- \frac 1p } q Z_p(f)^{1-n}.
\end{align*}
By Lemma \ref{GammaCentroid}, the $L_p$ Busemann-Petty centroid inequality \eqref{eqn:BPI}, \eqref{eqn:Zf}, \eqref{eqn:VolZ} and \eqref{eqn:Ez},

\begin{align*}
\mathcal{K}_{n,C_f} \int_{\R_+^n} C_f^*(\gradient f(y)) dy &\leq
d_1 \E_p(f)^{p-1} \left( \int_{\R^n_+} \left| \frac{\partial f}{\partial t} \right|^p dx dt \right)^{\frac{1}{p}}, \\
\end{align*}
where a careful computation of $d_1$ (that can be found in the appendix) gives the constant in the statement of the Theorem.

Finally, we obtain
\begin{eqnarray}
\label{lasteqs}
\left (\int_{\partial \R^n_+}|f(0,x)|^{\frac{p(n-1)}{n-p}}\; dx\right
)^{\frac{n-p}{n-1}} &\leq& \mathcal{K}_{n,C_f} \int_{\R^n_+} {C_f}^*(\gradient f) dxdt \\
&\leq& d_1 \E_p(f)^{p-1} \left( \int_{\R^n_+} \left| \frac{\partial f}{\partial t} \right|^p dx dt \right)^{\frac{1}{p}}. \nonumber
\end{eqnarray}
\qed\\

\section{Further comments}
\label{sec:end}

As mentioned in introduction, the same ideas used in \cite{L-Y-Z-1} to estimate $\E(f)$, involving H\"older's inequality and Fubini's theorem, produce $\E_p(f) \leq \|\tilde{\nabla} f\|_p$. Thus, using the Young inequality, we get

\begin{eqnarray}
	\label{lasteqs2}
	\left (\int_{\partial \R^n_+}|f(0,x)|^{\frac{p(n-1)}{n-p}}\; dx\right
)^{\frac{n-p}{n-1}} &\leq& d_2 \left( \E_p(f)^p + \left\| \dt f \right\|_p^p \right) \\
&\leq& d_2 \left( \int_{\R^n_+} | \tilde{\gradient} f|^p + \left|\dt f\right|^p dxdt \right) \nonumber
\end{eqnarray}
with $d_2 = (p^{\frac{1}{p}} q^{\frac{1}{q}})^{-\frac{1}{p}} d_1$.

Observe that the inequality $\E_p(f) \leq \|\tilde{\gradient} f\|_p$ becomes an equality for functions which are radial with respect to the variable $x$.
On the other hand, taking $C^*(t, x) = |t|^p + |x|^p$ with Euclidean norm in $x$, Theorem \ref{nazaret} provides the sharp inequality
\begin{eqnarray*}
\left (\int_{\partial \R^n_+}|f(0,x)|^{\frac{p(n-1)}{n-p}}\; dx\right
)^{\frac{n-p}{n-1}} &\leq& \mathcal{K}_{n,C} \int_{\R^n_+} {C}^*(\gradient f) dxdt\\
&\leq& \mathcal{K}_{n,C} \left( \int_{\R^n_+} \left|\dt f\right|^p  + | \tilde{\gradient} f|^p dxdt \right)
\end{eqnarray*}
and a careful calculation of the constant above using the formula of Lemma \ref{constant_L} gives exactly $d_2$.
This is the reason for the choice of the constants in Definition \ref{Cfdefinition}.
Since the extremal functions \eqref{extremalsC} corresponding to the previous inequality are radial with respect to $x$, we conclude that the inequalities \eqref{lasteqs2} and thus \eqref{lasteqs} are sharp.

Now if $f$ is an extremal function of equation \eqref{main_affine_trace_ineq-p} then the equality on $\vol(K_{f,0}) = \vol(\Gamma_p K_{f,0})$ implies that $K_{f,0}$ is an ellipsoid.
Thus after a $\gl_{n,+}$ change of coordinates, we may assume $K_{f,0}$ is a ball and $C_f^*$ is the function $C^*$ defined in the above paragraph.

For the special case $p=2$, the extremal functions where characterized by Escobar in \cite{Escobar}. Then the extremal functions of equation \eqref{main_affine_trace_ineq} are exactly the ones described in Theorem \ref{mainthm}.

Since the characterization of the extremal functions in the general case $1 < p < n$ was left open by Nazaret in \cite{Na} we cannot conclude that $f$ is of the form \eqref{extremalsC}.
But this is the only step left to characterize the extremal functions of \eqref{main_affine_trace_ineq-p}.

In the case the aforementioned characterization is done, there is an other important conclusion and that is that \eqref{main_affine_trace_ineq-p} is not stronger than the classical sharp trace inequality
\[\left (\int_{\partial \R^n_+}|f(0,x)|^{\frac{p(n-1)}{n-p}}\; dx\right
)^{\frac{n-p}{n-1}}\leq \mathcal{K}_{n,p} \int_{\R^n_+}|\nabla f(t,x)|^p\; dx dt\]
simply because they have different families of extremal functions.

\section{Appendix}

In this appendix, for sake of completness, we compute the constant $d_1$ that appears in the proof of Theorem \ref{mainthm-p}.

Firstly,

\begin{align*}
\mathcal{L}_{C_f}^p \int_{\R_+^n} C_f^*(\gradient f(y)) dy &=
\left( p q^{-\frac{1}{pq}} (n-p)^{-\frac{1}{q}}(q(n-1))^{-\frac 1 {q(n-1)}} \left(\frac{\Gamma \left(\frac{n-1}{p}\right) \Gamma \left(\frac{n-1}{q}+1\right)}{q \Gamma (n)}\right)^{-\frac{1}{q(n-1)}}
\right)^p\\
&\times \vol(K_{f,0})^{- \frac p {q(n-1) } } \alpha_f^{- \frac 1p } q Z_p(f)^{1-n}\\
& = p^p q^{-\frac{1}{q}} (n-p)^{-p/q} \left(\frac{(n-1)\Gamma \left(\frac{n-1}{p}\right) \Gamma \left(\frac{n-1}{q}+1\right)}{\Gamma (n)}\right)^{-\frac{p}{q(n-1)}}
\\
(\mbox{using } (\ref{eqn:VolKo}))  \quad &\times \left[\left(p q^{\frac pq } (n+p-1) a_{n-1,p}\right)^{\frac{n-1}{p}} \vol(L_f)^{\frac{n-1}p} \vol(\Gamma_p L_f)\right]^{- \frac p {q(n-1) } }\\
&\times \alpha_f^{- \frac 1p } q Z_p(f)^{1-n}\\
& =p^p q^{-\frac{1}{q}} (n-p)^{-p/q} \left(\frac{(n-1)\Gamma \left(\frac{n-1}{p}\right) \Gamma \left(\frac{n-1}{q}+1\right)}{\Gamma (n)}\right)^{-\frac{p}{q(n-1)}}
\\
 \quad &\times \left(p q^{\frac pq } (n+p-1) a_{n-1,p}\right)^{-\frac{1}{q}} \vol(L_f)^{\frac{-1}{q}} \vol(\Gamma_p L_f)^{- \frac p {q(n-1) } }\\
&\times \alpha_f^{- \frac 1p } q Z_p(f)^{1-n}\\
(\mbox {Using B-P ineq for }L_f)&\leq p^{p-\frac{1}{q}} q^{-\frac{1}{q}-\frac{p}{q^2}+1} (n-p)^{-p/q} \left(\frac{(n-1)\Gamma \left(\frac{n-1}{p}\right) \Gamma \left(\frac{n-1}{q}+1\right)}{\Gamma (n)}\right)^{-\frac{p}{q(n-1)}}
\\
 \quad &\times  \left((n+p-1) a_{n-1,p}\right)^{\frac{-1}{q}}  \vol( L_f)^{- \frac {p+n-1} {q(n-1) } }\\
&\times \alpha_f^{- \frac 1p } Z_p(f)^{1-n}\\
&=p^{p-\frac{1}{q}} q^{-\frac{1}{q}-\frac{p}{q^2}+1} (n-p)^{-p/q} \left(\frac{(n-1)\Gamma \left(\frac{n-1}{p}\right) \Gamma \left(\frac{n-1}{q}+1\right)}{\Gamma (n)}\right)^{-\frac{p}{q(n-1)}}
\\
 (\mbox{Using }(\ref{eqn:VolZ})) \quad &\times \left((n+p-1) a_{n-1,p}\right)^{\frac{-1}{q}}  \left(\frac{1}{n-1}c_{n-1,p}^{n-1}\right)^{-\frac {p+n-1} {q(n-1) } }\\&\times\E_p(f)^{\frac{n+p-1}{q}}\alpha_f^{- \frac 1p } Z_p(f)^{1-n}\\
 &=p^{p-\frac{1}{q}} q^{-\frac{1}{q}-\frac{p}{q^2}+1} (n-p)^{-p/q} \left(\frac{(n-1)\Gamma \left(\frac{n-1}{p}\right) \Gamma \left(\frac{n-1}{q}+1\right)}{\Gamma (n)}\right)^{-\frac{p}{q(n-1)}}
\\
 \quad &\times \left((n+p-1) a_{n-1,p}\right)^{\frac{-1}{q}}  \left(\frac{1}{n-1}c_{n-1,p}^{n-1}\right)^{-\frac {p+n-1} {q(n-1) } }\\
 (\mbox{By definition of }\alpha_f) &\times\E_p(f)^{\frac{n+p-1}{q}}\left( q  \left\| \dt f \right\|_\pmas^{-p} Z_p(f)^{1-n}\right)^{- \frac 1p }  Z_p(f)^{1-n}\\
 &=p^{p-\frac{1}{q}} q^{-\frac{p}{q^2}} (n-p)^{-p/q} \left(\frac{(n-1)\Gamma \left(\frac{n-1}{p}\right) \Gamma \left(\frac{n-1}{q}+1\right)}{\Gamma (n)}\right)^{-\frac{p}{q(n-1)}}
\\
 \quad &\times \left((n+p-1) a_{n-1,p}\right)^{\frac{-1}{q}}  \left(\frac{1}{n-1}c_{n-1,p}^{n-1}\right)^{-\frac {p+n-1} {q(n-1) } }\\
  &\times\E_p(f)^{\frac{n+p-1}{q}} \left\| \dt f \right\|_\pmas  Z_p(f)^{1-n+\frac{n-1}{p}}\\
&=p^{p-\frac{1}{q}} q^{-\frac{p}{q^2}} (n-p)^{-p/q} \left(\frac{(n-1)\Gamma \left(\frac{n-1}{p}\right) \Gamma \left(\frac{n-1}{q}+1\right)}{\Gamma (n)}\right)^{-\frac{p}{q(n-1)}}
\\
 \quad &\times \left((n+p-1) a_{n-1,p}\right)^{\frac{-1}{q}}  \left(\frac{1}{n-1}c_{n-1,p}^{n-1}\right)^{-\frac {p+n-1} {q(n-1) } }\\
 (\mbox{Using } (\ref{eqn:Ez})) &\times\E_p(f)^{\frac{n+p-1}{q}} \left\| \dt f \right\|_\pmas c_{n-1,p}^{(n-1)-\frac{n-1}{p}} \E_p(f)^{1-n+\frac{n-1}{p}}\\
   &=p^{p-\frac{1}{q}} q^{-\frac{p}{q^2}} (n-p)^{-p/q} \left(\frac{(n-1)\Gamma \left(\frac{n-1}{p}\right) \Gamma \left(\frac{n-1}{q}+1\right)}{\Gamma (n)}\right)^{-\frac{p}{q(n-1)}}
\\
\quad &\times \left(n+p-1\right)^{\frac{-1}{q}}\left( a_{n-1,p}\right)^{\frac{-1}{q}}  \left(\frac{1}{n-1}\right)^{-\frac {p+n-1} {q(n-1) } }\\
  &\times\left\| \dt f \right\|_\pmas c_{n-1,p}^{(n-1)-\frac{n-1}{p}-\frac {p+n-1} {q }}  \E_p(f)^{1-n+\frac{n-1}{p}+\frac{n+p-1}{q}}\\
  &=p^{p-\frac{1}{q}} q^{-\frac{p}{q^2}} (n-p)^{-p/q} \left(\frac{(n-1)\Gamma \left(\frac{n-1}{p}\right) \Gamma \left(\frac{n-1}{q}+1\right)}{\Gamma (n)}\right)^{-\frac{p}{q(n-1)}}
\\
\quad &\times \left(n+p-1\right)^{\frac{-1}{q}}\left( a_{n-1,p}\right)^{\frac{-1}{q}}  \left(\frac{1}{n-1}\right)^{-\frac {p+n-1} {q(n-1) } }c_{n-1,p}^{\frac{-p}{q}}\\
  &\times\left\| \dt f \right\|_\pmas  \E_p(f)^{\frac{p}{q}}.\\
\\
\end{align*}
Let us now denote
\begin{align*}
\mathcal{A}_{n,p} &= \frac{1}{p} \times p^{p-\frac{1}{q}} q^{-\frac{p}{q^2}} (n-p)^{-\frac{p}{q}} (n+p-1)^{-\frac{1}{q}} \left(\frac{1}{n-1}\right)^{-\frac{n+p-1}{(n-1) q}}\\
&\times   \left(\frac{(n-1) \Gamma \left(\frac{n-1}{p}\right) \Gamma \left(\frac{n-1}{q}+1\right)}{\Gamma (n)}\right)^{-\frac{p}{(n-1) q}} \left(a_{n-1,p} c_{n-1,p}^p\right){}^{-\frac{1}{q}}.
\end{align*}
But,

\begin{align*}
a_{n-1,p} c_{n-1,p}^p
&= \frac{\left(2^{-\frac{1}{p}} \left((n-1) \omega _{n-1}\right){}^{\frac{1}{n-1}} \left(\frac{(n-1) \omega _{n-1} \omega _{p-1}}{\omega _{n+p-3}}\right){}^{\frac{1}{p}}\right){}^p \omega _{n+p-1}}{\omega _2 \omega _{n-1} \omega _{p-1}}\\
&= \frac{(n-1) \left((n-1) \omega _{n-1}\right){}^{\frac{p}{n-1}} \omega _{n+p-1}}{2 \omega _2 \omega _{n+p-3}}\\
&= \frac{\pi  (n-1) \left((n-1) \omega _{n-1}\right){}^{\frac{p}{n-1}}}{\omega _2 (n+p-1)}\\
&= \frac{\pi  (n-1)^{\frac{n+p-1}{n-1}} \omega _{n-1}^{\frac{p}{n-1}}}{\omega _2 (n+p-1)}.
\end{align*}
Then we obtain that
\begin{align*}
\mathcal{A}_{n,p} &= p^{p - 1 -\frac{1}{q}} q^{-\frac{p}{q^2}} \left(\frac{1}{n-1}\right)^{-\frac{n+p-1}{(n-1) q}} (n-p)^{-\frac{p}{q}} (n+p-1)^{-\frac{1}{q}} \\
&\times  \left(\frac{(n-1) \Gamma \left(\frac{n-1}{p}\right) \Gamma \left(\frac{n-1}{q}+1\right)}{\Gamma (n)}\right)^{-\frac{p}{(n-1) q}} \left(\frac{\pi ^{p/2} (n-1)^{\frac{p}{n-1}+1} \Gamma \left(\frac{n+1}{2}\right)^{\frac{p}{1-n}}}{n+p-1}\right)^{-\frac{1}{q}}\\
&= p^{p - 1 -\frac{1}{q}} \pi ^{-\frac{p}{2 q}} q^{-\frac{p}{q^2}} (n-p)^{-\frac{p}{q}} \left(\frac{(n-1) \Gamma \left(\frac{n-1}{p}\right) \Gamma \left(\frac{n-1}{q}+1\right)}{\Gamma (n) \Gamma \left(\frac{n+1}{2}\right)}\right)^{-\frac{p}{(n-1) q}}\\
&= p^{\frac{1}{p} - 1} \pi ^{\frac{1}{2}-\frac{p}{2}} q^{\frac{1}{q}} \left(\frac{n-p}{p-1}\right)^{1-p} \left(\frac{\Gamma (n) \Gamma \left(\frac{n+1}{2}\right)}{(n-1) \Gamma \left(\frac{n-1}{p}\right) \Gamma \left(\frac{n}{q}+\frac{1}{p}\right)}\right)^{\frac{p-1}{n-1}}\\
&= \pi^{-\frac{p-1}{2}} \left(\frac{(p - 1)^{\frac{p-1}{p}}}{n - p}\right)^{p-1} \left(\frac{ \Gamma (n) \Gamma \left(\frac{n+1}{2}\right)}{(n-1)\Gamma \left(\frac{n-1}{p}\right) \Gamma \left(\frac{n (p-1)+1}{p}\right)}\right)^{\frac{p-1}{n-1}},
\end{align*}
\bigskip
where $d_1 = p \mathcal{A}_{n,p}$ is the optimal constant that appears in our main theorem.

\end{document}